\documentclass[dvipdfmx]{amsart}
\usepackage[margin=20truemm,tmargin=20truemm,lmargin=30truemm,rmargin=30truemm]{geometry}

\usepackage{amsmath}
\usepackage{amssymb}
\usepackage{mathrsfs}
\usepackage{mathtools}

\mathtoolsset{showonlyrefs, showmanualtags}

\usepackage{bm}
\usepackage{comment}
\usepackage{amsthm}
\usepackage{mleftright}
\usepackage{enumerate}
\usepackage{enumitem}
\usepackage{refcount}
\usepackage{xparse}
\newenvironment{eqenumerate}
{\begin{enumerate}[ref=\thesection.\theenumi]
		
		\setcounter{enumi}{\value{equation}}}
	{\setcounter{equation}{\value{enumi}}
\end{enumerate}}
\NewDocumentCommand\eqitem{ o }{
	\setcounter{enumi}{\value{equation}}
	\IfValueTF{#1}
	{\item[#1]}
	{\item}
	\setcounter{equation}{\value{enumi}}
}
\usepackage{physics}
\newtheorem{theorem}{Theorem}[section]
\newtheorem{corollary}[theorem]{Corollary}
\newtheorem{lemma}[theorem]{Lemma}

\theoremstyle{definition}
\newtheorem{assum}[theorem]{Assumption}

\newtheorem*{Acknow*}{Acknowledgement}
\numberwithin{equation}{section}

\renewcommand{\Re}{\operatorname{Re}}
\renewcommand{\S}{\mathbb{S}}
\DeclareMathOperator{\supp}{supp}

\newcommand{\R}{\mathbb{R}}
\newcommand{\N}{\mathbb{N}}
\newcommand{\C}{\mathbb{C}}
\newcommand{\Z}{\mathbb{Z}}
\newcommand{\w}{\mathit{w}}

\newcommand{\rad}{\textup{rad.}}
\newcommand{\indicator}[1]{\bm{1}_{#1}}
\renewcommand{\abs}[1]{\lvert#1\rvert} 
\newcommand{\Abs}[1]{\mleft \lvert #1 \mright \rvert}
\renewcommand{\norm}[1]{\lVert#1\rVert} 
 
\newcommand{\set}[2]{\{ \, #1 : #2 \, \} } 
\newcommand{\Set}[2]{ \mleft \{ \, #1 : #2 \, \mright\} } 
\newcommand{\measure}[1]{ \lvert#1\rvert }

\newcommand{\HP}{ \mathcal{H} }
\newcommand{\Pkn}[2]{P_{#1}^{#2}}
\newcommand{\Pk}[1]{\Pkn{#1}{}}

\DeclareMathOperator{\sgn}{ sgn }
\newcommand{\conjugate}{\overline}

\newcommand{\fpm}{ f_{\pm} }

\newcommand{\fk}[1]{f_{#1}}

\newcommand{\fkn}[2]{f_{#1}^{#2}}

\newcommand{\fknpm}[2]{f_{#1, \pm}^{#2}}
\newcommand{\fknp}[2]{f_{#1, +}^{#2}}
\newcommand{\fknm}[2]{f_{#1, -}^{#2}}

\newcommand\restr[2]{{
		\left.\kern-\nulldelimiterspace 
		#1 
		\vphantom{\big|} 
		\right|_{#2} 
}}
\renewcommand{\innerproduct}[2]{\langle #1, #2 \rangle}
\newcommand{\Innerproduct}[2]{\mleft \langle #1, #2 \mright \rangle}
\newcommand*{\variabledot}{\makebox[1ex]{\textbf{$\cdot$}}}
\newcommand{\adjoint}[1]{#1^{*}}
\newcommand{\Quadmatrix}{\Lambda}
\newcommand{\positiveR}{\R_{>0}}
\newcommand{\eigenspace}{W}
\newcommand{\DiracS}{\widetilde{S}}
\newcommand{\Diraclambda}{\widetilde{\lambda}}
\newcommand{\DiracE}{\widetilde{E}}

\usepackage[noadjust]{cite}
\usepackage[dvipdfmx]{hyperref}
\hypersetup{
	setpagesize=false,
	bookmarksnumbered=true,
	bookmarksopen=true,
	colorlinks=true,
	linkcolor=blue,
	citecolor=blue,
}

	\makeatletter
	\@addtoreset{equation}{section}
	
	\makeatother
	
	\title[Smoothing estimates for Dirac equations with radial data]{Optimal constants of smoothing estimates for Dirac equations with radial data}
	\date{}
	
	\author{Makoto Ikoma}
	\address[Makoto Ikoma]{Graduate School of Mathematics, Nagoya University, Furocho, Chikusaku, Nagoya, Aichi, 464-8602, Japan}
	\email{ikma.m18005d@gmail.com}
	
	\author{Soichiro Suzuki}
	\address[Soichiro Suzuki]{Department of Mathematics, Chuo University, 1-13-27, Kasuga, Bunkyo-ku, Tokyo, 112-8551, Japan}
	\email{soichiro.suzuki.m18020a@gmail.com}
	\thanks{The second author was supported by Japan Society for the Promotion of Science (JSPS) KAKENHI Grant Number JP23KJ1939.}
	
	
	
\begin{document}
				\begin{abstract}
			Kato--Yajima smoothing estimates are one of the fundamental results in study of dispersive equations such as Schr\"odinger equations and Dirac equations.
			For $d$-dimensional Schr\"odinger-type equations ($d \geq 2$), optimal constants of smoothing estimates were obtained by Bez--Saito--Sugimoto (2017) via the so-called Funk--Hecke theorem. 
			Recently Ikoma (2022) considered optimal constants for $d$-dimensional Dirac equations using a similar method, 
			and it was revealed that determining optimal constants for Dirac equations is much harder than the case of Schr\"odinger-type equations. 
			Indeed, Ikoma obtained the optimal constant in the case $d = 2$, but only upper bounds (which seem not optimal) were given in other dimensions.
			In this paper, we give optimal constants for $d$-dimensional Schr\"odinger-type and Dirac equations with radial initial data for any $d \geq 2$.
			In addition, we also give optimal constants for the one-dimensional Schr\"odinger-type and Dirac equations.
		\end{abstract}
		
		\maketitle

		\section{Introduction}
		Kato--Yajima smoothing estimates are one of the fundamental results in study of dispersive equations such as Schr\"odinger equations and Dirac equations, which were firstly observed by Kato \cite{Kato2} and have been studied by numerous researchers. For example, see \cite{BD1,BD2,BK,BS,BSS,Chihara,Con_Saut,Hos1,Hos2,KPV,KPV2,KPV3,KPV4,KPV5,KPV6,KY,Lin_Pon,Ruz_sug,Sjo,Simon,Sugimoto1,Sugimoto2,Vega,Vilela,Walther1,Watanabe} for Schr\"{o}dinger equations and \cite{BU,DF,Ikoma,KOY,Kato1} for Dirac equations. 
		In particular, Bez--Saito--Sugimoto \cite{BSS},  Bez--Sugimoto \cite{BS}, Ikoma \cite{Ikoma}, Simon \cite{Simon} and Watanabe \cite{Watanabe} considered optimal constants and extremisers of smoothing estimates (see \cite[Introduction]{Ikoma} for overviews of these results). 
		Our aim in this paper is to give optimal constants and extremisers of smoothing estimates for Dirac equations with radial initial data.
		
		In 2015, Bez--Saito--Sugimoto \cite{BSS} pointed out that the so-called Funk--Hecke theorem (see Theorem \ref{thm:Funk-Hecke}) is quite useful to obtain optimal constants of smoothing estimates and to determine extremisers of the inequalities for the following equation:
		\begin{equation} \label{eq:Schrodinger}
			\begin{cases}
				i \partial_t u(x, t) = \phi(\abs{D}) u(x, t) , & (x, t) \in \R^d \times \R , \\
				u(x, 0) = f(x) , & x \in \R^d , 
			\end{cases}
		\end{equation}
		where $\phi(\abs{D})$ denotes the Fourier multiplier operator whose symbol is $\phi(\abs{\variabledot})$, that is, 
		\begin{equation}
			\mathcal{F} \phi(\abs{D}) f = \phi(\abs{\xi}) \widehat{f}(\xi) .
		\end{equation}
		Particularly, \eqref{eq:Schrodinger} becomes the free Schr\"odinger equation if $\phi(r) = r^2$ and the relativistic Schr\"odinger equation if $\phi(r) = \sqrt{r^2 + m^2}$, respectively.
		For convenience, we refer the equation \eqref{eq:Schrodinger} as a Schrodinger-type equation.
		The smoothing estimate of the Schrodinger-type equation is expressed as
		\begin{equation}
			\int_{\R}\int_{\R^d}\abs{\psi(\abs{D}) e^{it\phi(\abs{D})}f(x)}^2 \w(\abs{x}) \, dx \, dt \leq C\norm{f}^2_{L^2(\R^d)} , \label{eq:smoothing Schrodinger}
		\end{equation}
		here functions $\w$ and $\psi$ are called spatial weight and dispersion relation, respectively.
		We write $C_{d}(\w,\psi,\phi)$ for the optimal constant for the inequality \eqref{eq:smoothing Schrodinger}, in other words,
		\begin{equation}
			C_{d}(\w,\psi,\phi) \coloneqq \sup_{\substack{f\in L^2(\R^d) \\ 
					f\not\equiv0}}\int_{\R}\int_{\R^d}\abs{\psi(\abs{D})e^{it\phi(\abs{D})}f(x)}^2\w(\abs{x}) \, dx \, dt / \norm{f}^2_{L^2(\R^d)} . \label{eq:optimal Schrodinger}
		\end{equation}
		Under some reasonable assumptions for $\w$, $\psi$ and $\phi$ (see Assumption \ref{assumption} for details), \cite{BSS} established the following result:
		\begin{theorem}[{\cite[Theorem 1.1]{BSS}}] \label{thm:Schrodinger}
			Let $d\geq2$. 
			Then we have 
			\begin{equation}
				C_{d}(\w,\psi,\phi) = \frac{1}{(2\pi)^{d-1}}\sup_{k\in\N_{0}}\sup_{r>0}\lambda_{k}(r), \label{optimal_answer_1}
			\end{equation}
			where
			\begin{equation}
				\lambda_{k}(r)\coloneqq \abs{\S^{d-2}}\frac{r^{d-1}\psi(r)^2}{\abs{\phi^{\prime}(r)}}\int_{-1}^{1}F_{\w}(r^2(1-t))p_{d,k}(t)(1-t^2)^{\frac{d-3}{2}} \, dt. \label{eq:lambda_k}
			\end{equation}
			Here $F_{\w}$ is defined by the relation
			\begin{equation}
				\widehat{\w(\abs{\,\cdot\, })}(\xi) = F_{\w}(\tfrac{1}{2}\abs{\xi}^2)  \label{w_Fourier}
			\end{equation}
			and $p_{d,k}$ is the Legendre polynomial of degree $k$ in $d$ dimensions, which may be defined in a number of ways, for example, via the Rodrigues formula,
			\begin{equation}
				(1-t^2)^{\frac{d-3}{2}}p_{d,k}(t) = (-1)^{k}\frac{\Gamma(\frac{d-1}{2})}{2^k\Gamma(k+\frac{d-1}{2})}
				\frac{d^k}{dt^k}(1-t^2)^{k+\frac{d-3}{2}}. \label{Legendre}
			\end{equation}
		\end{theorem}
		Since we are interested in explicit constants, here we clarify that the Fourier transform in this paper is defined by  
		\begin{equation}
			\hat{f}(\xi) \coloneqq \int_{x \in \R^d}f(x)e^{-ix\cdot\xi} \, dx .
		\end{equation}
		In this case, the Plancherel theorem states that 
		\begin{equation}
			\norm{\hat{f}}_{L^2(\R^d)}^2 = (2 \pi)^d \norm{f}_{L^2(\R^d)}^2 .
		\end{equation}
		
		Now we are going to discuss the free Dirac equation.	
		Let $d \geq 1$ and write $N \coloneqq 2^{\lfloor(d+1)/2\rfloor}$. 
		The $d$-dimensional free Dirac equation with mass $m \geq 0$ is given by 
		\begin{equation} \label{eq:Dirac}
			\begin{cases}
				i\partial_{t} u(x,t) = H_m u(x, t), & (x, t) \in \R^d \times \R , \\
				u(x,0) = f(x) , & x \in \R^d .
			\end{cases}
		\end{equation}
		Here the Dirac operator $H_m$ is defined by
		\begin{equation}
			H_m \coloneqq \alpha\cdot D + m\beta = \sum_{j=1}^{d}\alpha_jD_j + m\beta ,
		\end{equation}
		where $\alpha_1, \alpha_2, \ldots, \alpha_d$, $\alpha_{d+1}=\beta$
		are $N\times N$ Hermitian matrices satisfying the anti-commutation relation $\alpha_j\alpha_k + \alpha_k\alpha_j = 2\delta_{jk}I_N$. 
		In this case, the smoothing estimate is 
		\begin{equation} \label{eq:smoothing Dirac}
			\int_{\R}\int_{\R^d}\abs{\psi(\abs{D})e^{-itH_m}f(x)}^2\w(\abs{x}) \, dx \, dt \leq C\norm{f}^2_{L^2(\R^d,\C^N)}, 
		\end{equation}
		and let $\widetilde{C}_{d}(\w, \psi, m)$ be the optimal constant of the inequality \eqref{eq:smoothing Dirac}:
		\begin{equation}
			\widetilde{C}_{d}(\w,\psi,m)\coloneqq \hspace{-12pt}\sup_{\substack{f\in L^2(\R^d,\C^N) \\ 
					f\not\equiv0}}\int_{\R}\int_{\R^d}\abs{\psi(\abs{D})e^{-itH_m}f(x)}^2\w(\abs{x}) \, dx \, dt /\norm{f}^2_{L^2(\R^d,\C^N)} . \label{eq:optimal Dirac}
		\end{equation}
		Recently, Ikoma \cite{Ikoma} established the following results:
		\begin{theorem}[{\cite[Theorem 2.2]{Ikoma}}]\label{thm:2D Dirac}
			Let $d = 2$. 
			Then we have
			\begin{equation} 
				\widetilde{C}_{2}(\w,\psi,m) = \frac{1}{2\pi}\sup_{ k \in \N }\sup_{r>0}\Diraclambda_{k}(r),
			\end{equation}
			where
			\begin{equation}
				\Diraclambda_{k}(r) 
				\coloneqq \frac{1}{2} \mleft( \lambda_{k}(r) + \lambda_{k+1}(r) + \frac{m}{\sqrt{r^2+m^2}} \abs{ \lambda_{k}(r) - \lambda_{k+1}(r) } \mright)
			\end{equation}
			and $\lambda_{k}(r)$ is \eqref{eq:lambda_k} with $\phi(r)=\sqrt{r^2+m^2}$.
		\end{theorem}
		\begin{theorem}[{\cite[Theorem 2.1]{Ikoma}}]\label{thm:Dirac upper bound}
			Let $d \geq 3$. 
			Then we have
			\begin{equation} 
				\widetilde{C}_{d}(\w,\psi,m) \leq C_{d}(\w,\psi,\phi) = \frac{1}{(2\pi)^{d-1}}\sup_{ k \in \N }\sup_{r>0} \lambda_{k}(r),
			\end{equation}
			where $\lambda_{k}(r)$ is \eqref{eq:lambda_k} with $\phi(r)=\sqrt{r^2+m^2}$.
			In other words, the optimal constant for the Dirac equation is less than or equal to that for the relativistic Schr\"odinger equation.
		\end{theorem}
		Unlike Schr\"odinger-type equations, the optimal constant is known only in the case $d = 2$ for Dirac equations.
		In particular, the physically most important case $d = 3$ remains open.
		Also, though Theorem \ref{thm:Dirac upper bound} gives upper bounds, there are no known results on lower bounds.
		The main difficulty here is that 
		the multiplier of Dirac operator $H_m$ is not radial.
		Indeed, it causes a large number (growing exponentially with respect to the dimension $d$) of cross terms to appear when computing the $L^2$ norm, which does not exist at all in the case of the Schr\"odinger-type equation.
		
		In this paper, we consider the smoothing estimate of $d$-dimensional Dirac equations with radial initial data. 
		In that case cross terms does not appear, and actually we can obtain optimal constants. 
		Note these optimal constants with the radial restriction serve as lower bounds of constants without the restriction.
		Additionally, we also give optimal constants of the one-dimensional Schr\"odinger-type equation and Dirac equation. 
		You may see that in the case $d = 1$, the factor $\S^{d-2}$ in the definition of $\lambda_k$ given by \eqref{eq:lambda_k}, the surface area of the $(d-2)$-dimensional unit sphere, does not make sense anymore. 
		Nevertheless, by redefining it appropriately, we establish results in the case $d = 1$ using essentially the same argument as that of Theorems \ref{thm:Schrodinger} and \ref{thm:2D Dirac}.
		
		This paper is organized as follows.
		
		In Section \ref{subsection:main results Schrodinger} and Section \ref{subsection:main results Dirac}, we state results regarding Schr\"odinger-type equations and Dirac equations, respectively. 
		
		In Section \ref{section:preliminaries}, we introduce two classical results on spherical harmonics: the spherical harmonics decomposition and the Funk--Hecke theorem. 
		They lead us to Lemma \ref{lem:Sf norm decomposition Schrodinger} and Lemma \ref{lem:Sf norm decomposition Dirac}, which were established by \cite{BSS} and \cite{Ikoma}, respectively. 
		These lemmas played an important role in their respective papers, and will again do in the present paper.
		Here we remark that the spherical harmonics decomposition and the Funk--Hecke theorem are usually considered in the case $d \geq 2$, 
		but they are also true in the case $d = 1$. 
		
		In Section \ref{section:proof}, we prove main results.
		Though the proofs for the Schr\"odinger-type equation are essentially the same as that of \cite[Theorem 1.1]{BSS}, 
		we give details for comparison with the proofs for the Dirac equation. 
		In fact, you may see that the proofs for the Dirac equation are much more complicated.
		
		\section{Main results} \label{section:main results}
		In this section, we discuss the main results. 
		Through out this paper, we always suppose that the following Assumption \ref{assumption} holds:
		\begin{assum} \label{assumption}
			The spatial weight $\w$, the smoothing function $\psi$, and the dispersion relation $\phi$ satisfy the following conditions.
			\begin{description}
				\item[The spatial weight]
				\phantom{a}
				\begin{itemize}
					\item In the case $d = 1$, 
					the spatial weight $\w(\abs{\,\cdot\,})\in L^1(\R)$ is positive and even, in which case we write
					\begin{equation}
						\widehat{\w(\abs{\,\cdot\,})}(\xi) = F_{\w}(\tfrac{1}{2}\abs{\xi}^2) \tag*{\eqref{w_Fourier}}
					\end{equation}
					for some $F_{\w}:[0, \infty)\to\R$ that is uniformly continuous and bounded on $[0, \infty)$.
					\item In the case $d \geq 2$,
					the spatial weight $\w(\abs{\,\cdot\,})$ is positive, radial, and its Fourier transform (in the sense of tempered distributions) coincides some locally bounded function on $\R^d \setminus \{0\}$, in which case we write
					\begin{equation}
						\widehat{\w(\abs{\,\cdot\,})}(\xi) = F_{\w}(\tfrac{1}{2}\abs{\xi}^2) . \tag*{\eqref{w_Fourier}}
					\end{equation}
					Furthermore, $F_w$ is sufficiently nice to guarantee the continuity of $\lambda_{k} \colon \R_{>0} \to \R_{>0}$ for each $k\in\N_{0}$, here recall that
					\begin{equation}
						\lambda_{k}(r)\coloneqq \abs{\S^{d-2}}\frac{r^{d-1}\psi(r)^2}{\abs{\phi^{\prime}(r)}}\int_{-1}^{1}F_{\w}(r^2(1-t))p_{d,k}(t)(1-t^2)^{\frac{d-3}{2}} \, dt. \tag{\ref{eq:lambda_k}}
					\end{equation}
				\end{itemize}
				\item[The smoothing function]
				\phantom{a}
				\begin{itemize}
					\item In the case $d=1$, the smoothing function $\psi \colon (0, \infty) \to \R$ is continuous and non-negative.
					\item In the case $d=2$, the smoothing function $\psi \colon (0, \infty) \to \R$ is infinitely differentiable and non-negative.
				\end{itemize}
				\item[The dispersion relation]
				\phantom{a}
				\begin{itemize}
					\item 
					For Sch\"odinger equations, 
					the dispersion relation $\phi \colon (0, \infty) \to \R$ is continuously differentiable and injective.
					\item
					For Dirac equations, 
					the dispersion relation $\phi \colon (0, \infty) \to \R$ is given by $\phi(r) \coloneqq \sqrt{r^2 + m^2}$. Note that it is continuously differentiable and injective.
				\end{itemize}
			\end{description}
		\end{assum}
		\subsection{Schr\"odinger-type equations} \label{subsection:main results Schrodinger}
		In this section, we discuss Schr\"odinger-type equations. 
		Define the linear operator $S \colon L^2(\R^d) \to L^2(\R^{d+1})$ by
		\begin{equation}
			Sf(x,t)\coloneqq \w(\abs{x})^{1/2}\int_{\xi \in \R^d} e^{ix \cdot \xi} \psi(\abs{\xi}) e^{ i t \phi(\abs{\xi}) } f(\xi) \, d\xi ,
		\end{equation}
		in other words,
		\begin{equation}
			S\widehat{f}(x, t) = \abs{\psi(\abs{D}) e^{it\phi(\abs{D})}f(x)} \w(\abs{x})^{1/2} .
		\end{equation}
		Then it is obvious that the smoothing estimate \eqref{eq:smoothing Schrodinger} is equivalent to 
		\begin{equation}
			\norm{S \hat{f}}_{L^2(\R^{d+1})}^2 \leq C \norm{f}_{L^2(\R^d)}^2 = (2 \pi)^d C \norm{\hat{f}}_{L^2(\R^d)}^2 ,
		\end{equation}
		and therefore it is enough to consider the operator norm of $S$.
		
		At first we discuss the one dimensional case. 
		For simplicity, let
		\begin{equation}
			\norm{S} \coloneqq \norm{S}_{L^2(\R) \to L^2(\R^2)} .
		\end{equation}
		\begin{theorem}\label{thm:1D Schrodinger}
			Let $d = 1$ and write
			\begin{gather}
				\lambda_k(r) \coloneqq 
				\begin{dcases}
					\frac{\psi(r)^2}{\abs{ \phi^{\prime}(r) } } ( \norm{\w}_{L^1(\R)} + F_{\w}(2r^2) ) , & k = 0 , \\
					\frac{\psi(r)^2}{\abs{ \phi^{\prime}(r) } } ( \norm{\w}_{L^1(\R)} - F_{\w}(2r^2) ) , & k = 1 , 
				\end{dcases}  
				\label{eq:lambda_k 1D}
				\\
				\lambda^* 
				\coloneqq \sup_{r > 0} \max_{k = 0, 1} \lambda_k(r) 
				= \sup_{r > 0} \frac{\psi(r)^2}{\abs{ \phi^{\prime}(r) } } ( \norm{\w}_{L^1(\R)} + \abs{ F_{\w}(2r^2) } ) , \\
				E_k \coloneqq \set{ r>0 }{ \lambda_k(r) = \lambda^* } .
			\end{gather}
			Then, under Assumption \ref{assumption}, we have
			\begin{equation} \label{optimal_answer_3}
				\norm{S}^2 = 2 \pi \lambda^* .
			\end{equation}
			
			Regarding extremisers, fix $f \in L^2(\R)$ and decompose it as
			\begin{equation}
				f(\xi) = \fk{0}(\abs{\xi}) + \sgn (\xi) \fk{1}(\abs{\xi}) .
			\end{equation}
			Then the following are equivalent:
			\begin{eqenumerate}
				\eqitem \label{item:extremiser 1D Schrodinger}
				The equality 
				\begin{equation}
					\norm{Sf}_{L^2(\R^2)}^2 = 2 \pi \lambda^* \norm{f}^2_{L^2(\R)} 
				\end{equation}
				holds.
				\eqitem \label{item:extremiser condition 1D Schrodinger}
				The functions $\{ \fk{k} \}_{k = 0, 1}$ satisfy
				\begin{equation}
					\supp \fk{k} \subset E_k 
				\end{equation}
				for each $k = 0, 1$.
			\end{eqenumerate}
			As a consequence, extremisers exist if and only if there exists 
			$k \in \{ 0, 1 \}$ such that $\measure{E_k} > 0$.
		\end{theorem}
		Next we consider smoothing estimates with radial data. 
		For simplicity we write
		\begin{equation}
			\norm{S}_{\rad} \coloneqq \norm{S}_{L^2_\rad(\R^d) \to L^2(\R^{d+1})} ,
		\end{equation}
		here $L^2_\rad(\R^d)$ denotes the subspace of $L^2(\R^d)$ defined by
		\begin{equation}
			L^2_{\rad}(\R^d) \coloneqq \set{ f\in L^2(\R^d)}{ \text{$f$ is radial} }.
		\end{equation}
		\begin{theorem}\label{thm:radial Schrodinger}
			Let $d \geq 1$ and write
			\begin{gather}
				\lambda_\rad(r) \coloneqq \lambda_0(r) , \\
				\lambda_\rad^* \coloneqq \sup_{r > 0} \lambda_\rad(r) , \\
				E_{\rad} \coloneqq \set{ r>0 }{ \lambda_\rad(r) = \lambda_\rad^* }.
			\end{gather}
			Then, under Assumption \ref{assumption}, we have
			\begin{equation}
				\norm{S}_{\rad}^2 = 2 \pi \lambda_\rad^* .
			\end{equation}
			
			Regarding extremisers, 
			fix $f \in L^2_\rad(\R^d)$ and write
			\begin{equation}
				f(\xi) = \abs{\xi}^{-(d-1)/2} \fk{0}(\xi) .
			\end{equation}
			Then the following are equivalent:
			\begin{eqenumerate}
				\eqitem \label{item:extremiser radial Schrodinger}
				The equality 
				\begin{equation}
					\norm{Sf}_{L^2(\R^{d+1})}^2 = 2 \pi \lambda_\rad^* \norm{f}^2_{L^2(\R^d)}
				\end{equation}
				holds.
				\eqitem \label{item:extremiser condition radial Schrodinger}
				The function $\fk{0} \in L^2(\positiveR)$ satisfies
				\begin{equation}
					\supp{ \fk{0} } \subset E_\rad .
				\end{equation}
			\end{eqenumerate}
			As a consequence, extremisers exist if and only if $\measure{E_{\rad}} > 0$.
		\end{theorem}
		\subsection{Dirac equations} \label{subsection:main results Dirac}
		As in the case of Schr\"odinger-type equations, we define the linear operator $\DiracS \colon L^2(\R^d, \C^N) \to L^2(\R^{d+1}, \C^N)$ by 
		\begin{equation}
			\DiracS f(x,t) \coloneqq \w(\abs{x})^{1/2} \int_{\R^d} e^{ix \cdot \xi}\psi(\abs{\xi})e^{-itA_{\xi}}f(\xi)\, d\xi ,
		\end{equation}
		where
		\begin{equation}
			A_\xi \coloneqq \alpha\cdot \xi + m\beta = \sum_{j=1}^{d}\alpha_j \xi_j + m\beta ,
		\end{equation}
		and we consider the operator norms of $\DiracS$:
		\begin{gather}
			\norm{\DiracS} \coloneqq \norm{\DiracS}_{L^2(\R^d, \C^N) \to L^2(\R^{d+1}, \C^N)} , \\
			\norm{\DiracS}_{\rad} \coloneqq \norm{\DiracS}_{L^2_\rad(\R^d, \C^N) \to L^2(\R^{d+1}, \C^N)} .
		\end{gather}
		
		At first we discuss the one dimensional case. 
		\begin{theorem}\label{thm:1D Dirac}
			Let $d = 1$ and write
			\begin{gather}
				\Diraclambda(r) \coloneqq \frac{\psi(r)^2}{ \abs{ \phi^{\prime}(r) } } 
				\mleft( \norm{\w}_{L^1(\R)} + \frac{m}{\phi(r)} \abs{ F_{\w}(2r^2) } \mright) \label{eq:lambda 1D Dirac} , \\
				\Diraclambda^* \coloneqq \sup_{r>0} \Diraclambda(r) , \\
				\DiracE \coloneqq \set{ r>0 }{ \Diraclambda(r) = \Diraclambda^* }.
			\end{gather}
			Then, under Assumption \ref{assumption}, we have
			\begin{equation} \label{optimal_answer_4}
				\norm{\DiracS}^2 = 2 \pi \Diraclambda^* .
			\end{equation}
			
			Regarding extremisers, 
			fix $f \in L^2(\R, \C^2)$ and decompose it as
			\begin{equation}
				f(\xi) = \fk{0}(\abs{\xi}) + \sgn (\xi) \fk{1}(\abs{\xi}) .
			\end{equation}
			Then the following are equivalent:
			\begin{eqenumerate}
				\eqitem \label{item:extremiser 1D Dirac}
				The equality 
				\begin{equation}
					\norm{\DiracS f}_{L^2(\R^2, \C^2)}^2 = 2 \pi \lambda^* \norm{f}^2_{L^2(\R, \C^2)} 
				\end{equation}
				holds.
				\eqitem \label{item:extremiser condition 1D Dirac} 
				The functions $\{ \fk{k} \}_{k = 0, 1}$ satisfy
				\begin{equation}
					\supp \fk{k} \subset \DiracE
				\end{equation}
				for each $k = 0, 1$ and
				\begin{equation}
					\begin{pmatrix}
						\beta \fk{0}(r) \\
						\alpha \fk{1}(r) 
					\end{pmatrix}
					\in \eigenspace(r)
				\end{equation}
				for almost every $r > 0$,
				where $\eigenspace(r)$ is the linear subspace of $\C^4$ defined as follows:
				\begin{equation} \label{eq:eigenspace 1D}
					\eigenspace(r) = 
					\begin{dcases}
						\C^4, & m F_\w(2r^2) = 0 , \\
						\Set{ 
							s
							\begin{pmatrix}
								m + \phi(r) \\
								0 \\
								r \\
								0
							\end{pmatrix}
							+ 
							t 
							\begin{pmatrix}
								0 \\
								m + \phi(r) \\
								0 \\
								r 
							\end{pmatrix}
						}
						{ s, t \in \C } , 
						& m F_\w(2r^2) > 0 , \\
						\Set{ 
							s
							\begin{pmatrix}
								m - \phi(r) \\
								0 \\
								r \\
								0
							\end{pmatrix}
							+ 
							t 
							\begin{pmatrix}
								0 \\
								m - \phi(r) \\
								0 \\
								r 
							\end{pmatrix}
						}
						{ s, t \in \C } , 
						& m F_\w(2r^2) < 0 .
					\end{dcases}
				\end{equation}
			\end{eqenumerate}
			As a consequence, extremisers exist if and only if $\measure{\DiracE} > 0$.
		\end{theorem}
	For the case with radial data, we have the following result:
		\begin{theorem}\label{thm:radial Dirac}
			Let $d \geq 2$ and write
			\begin{gather}
				\Diraclambda_\rad(r)\coloneqq \frac{1}{2}\left\{\left(1+\frac{m^2}{\phi(r)^2}\right)\lambda_0(r)+\frac{r^2}{\phi(r)^2}\lambda_1(r)\right\}, \label{eq:lambda radial Dirac} \\
				\Diraclambda_\rad^* \coloneqq \sup_{r > 0} \Diraclambda_\rad(r) , \\
				\DiracE_{\rad}\coloneqq \set{ r>0 }{ \Diraclambda_\rad(r) = \Diraclambda_\rad^* }.
			\end{gather}
			Then, under Assumption \ref{assumption}, we have
			\begin{equation}
				\norm{\DiracS}_{\rad}^2 = 2 \pi \Diraclambda_\rad^* . \label{optimal_answer_5}
			\end{equation}
			
			Regarding extremisers, fix $f \in L^2_\rad(\R^d, \C^N)$ 
			and write
			\begin{equation}
				f(\xi) = r^{-(d-1)/2} \fk{0}(\xi) .
			\end{equation}
			Then the following are equivalent:
			\begin{eqenumerate}
				\eqitem \label{item:extremiser radial Dirac}
				The equality 
				\begin{equation}
					\norm{\DiracS f}_{L^2(\R^{d+1}, \C^N)}^2 = 2 \pi \Diraclambda_\rad^* \norm{f}^2_{L^2(\R^d,\C^N)}
				\end{equation}
				holds.
				\eqitem \label{item:extremiser condition radial Dirac}
				The function $\fk{0} \in L^2(\positiveR, \C^N)$ satisfies
				\begin{equation}
					\supp{ \fk{0} } \subset \DiracE_{\rad} .
				\end{equation}
			\end{eqenumerate}
			As a consequence, extremisers exist if and only if $\measure{\DiracE_{\rad}} > 0$.
		\end{theorem}
		Note that combining the trivial inequality $\norm{\DiracS}_\rad \leq \norm{\DiracS}$ and Theorem \ref{thm:Dirac upper bound} implies the following:
		\begin{corollary} \label{cor:lower and upper bounds}
			Let $d \geq 2$. 
			Then, under Assumption \ref{assumption}, we have
			\begin{equation}
				2 \pi \Diraclambda_\rad^* \leq \norm{\DiracS}^2 \leq 2 \pi \lambda^*. 
			\end{equation}
		\end{corollary}
		Using this, we obtain the following result:
	\begin{theorem} \label{thm:explicit value} \newcommand{\Diraclambdasup}{\Diraclambda^*}
	Let $d \geq 2$, $m > 0$, $1 < s < d$, and
	\begin{equation}
	 w(r) = r^{-s}, \quad \psi(r) = r^{1 - s/2} (r^2+m^2)^{-1/4} .
	\end{equation}
	Then we have
	\begin{equation}
		\norm{\DiracS}^2 = 2^{1-s} (2\pi)^{d+1} \frac{ \Gamma(s-1) \Gamma((d-s)/2) }{ ( \Gamma(s/2) )^2 \Gamma( (d + s)/2 - 1 ) } .
	\end{equation}	
\end{theorem}
		\section{Preliminaries} \label{section:preliminaries}
		\subsection{Spherical harmonics decomposition of \texorpdfstring{$L^2(\R^d)$}{L2(Rd)}} \label{subsection:Spherical harmonics decomposition}
		For each $k \in \N$, let $\HP_{k}(\R^d)$ and $\{ \Pkn{k}{n} \}_{1 \leq n \leq d_k}$ be the space of homogeneous harmonic polynomials of degree $k$ on $\R^d$ and its orthonormal basis, respectively.
		Furthermore, we write $\HP(\R^d) \coloneqq \bigoplus_{k \in \N} \HP_k(\R^d)$.
		Here the inner product of $P, Q \in \HP(\R^d)$ is defined by
		\begin{equation}
			\innerproduct{P}{Q}_{\HP(\R^d)}
			\coloneqq \innerproduct{ \restr{P}{\S^{d-1}} }{ \restr{Q}{\S^{d-1}} }_{L^2(\S^{d-1})}
			= \int_{ \theta \in \S^{d-1} } P(\theta) \conjugate{Q(\theta)} \, d\sigma(\theta) ,
		\end{equation}
		as usual.
		The spherical harmonics decomposition (Theorem \ref{thm:spherical harmonics decomposition}) and the Funk--Hecke theorem (Theorem \ref{thm:Funk-Hecke}) are well-known and play essential roles in \cite{BSS} and \cite{Ikoma}.
		\begin{theorem} \label{thm:spherical harmonics decomposition}
			For any $f \in L^2(\R^d)$, there exists $\{ \fkn{k}{n} \}_{ k \in \N, 1 \leq n \leq d_k } \subset L^2(\positiveR)$ satisfying
			\begin{gather}
				f(\xi) = \abs{\xi}^{-(d-1)/2} \sum_{k=0}^{\infty} \sum_{n=1}^{d_k} \Pkn{k}{n}(\xi / \abs{\xi}) \fkn{k}{n}(\abs{\xi}) , \label{eq:spherical harmonics decomposition} \\
				\norm{f}_{L^2(\R^d)}^2 = \sum_{k=0}^{\infty} \sum_{n=1}^{d_k}  \norm{\fkn{k}{n}}_{L^2(\positiveR)}^2 . \label{eq:spherical harmonics decomposition norm}
			\end{gather}
		\end{theorem}
		Note that in the case $d = 1$, we have
		\begin{equation} \label{eq:HPk}
			\HP_k(\R) = \begin{cases}
				\set{ c }{ c \in \C } , & k = 0 , \\
				\set{ c\xi }{ c \in \C } , & k = 1 , \\
				\{ 0 \} , & k \geq 2
			\end{cases}
		\end{equation}
		and therefore \eqref{eq:spherical harmonics decomposition} and \eqref{eq:spherical harmonics decomposition norm} are reduced to
		\begin{gather}
			f(\xi) = \frac{1}{\sqrt{2}} \fk{0}(\abs{\xi}) + \frac{ \sgn(\xi) }{\sqrt{2}} \fk{1}(\abs{\xi}), \label{eq:spherical harmonics decomposition d=1} \\
			\norm{f}_{L^2(\R)}^2 = \norm{\fk{0}}_{L^2(\positiveR)}^2 + \norm{\fk{1}}_{L^2(\positiveR)}^2 . \label{eq:spherical harmonics decomposition norm d=1}
		\end{gather}
		We remark that \eqref{eq:spherical harmonics decomposition d=1} is equivalent to the even--odd decomposition
		\begin{equation}
			f(\xi) = \frac{ f(\xi) + f(-\xi) }{2} + \frac{ f(\xi) + f(-\xi) }{2} .
		\end{equation}
		\subsection{The Funk--Hecke theorem} \label{subsection:Funk-Hecke}
		The Funk--Hecke theorem is as follows:
		\begin{theorem} \label{thm:Funk-Hecke}
			Let $d \geq 2$, $k \in \N$, $F \in L^1 ( [-1, 1], (1-t^2)^{(d-3)/2} \, dt )$ and write
			\begin{equation} \label{eq:mu in Funk-Hecke}
				\mu_{k}[F] \coloneqq \measure{\S^{d-2}} \int_{-1}^1 F(t) p_{d, k}(t) 
				(1-t^2)^{\frac{d-3}{2}} \, dt, 
			\end{equation}
			here recall that $p_{d, k}$ denotes the Legendre polynomial of degree $k$ in $d$ dimensions (see \eqref{Legendre}).
			Then, for any $P \in \HP_k(\R^d)$ and $\omega \in \S^{d-1}$, we have
			\begin{equation} \label{eq:Funk-Hecke}
				\int_{\theta \in \S^{d-1}} F(\theta \cdot \omega) P(\theta) \, d\sigma(\theta) 
				= \mu_{k}[F] P(\omega) . 
			\end{equation}
		\end{theorem}
		Note that the function $\lambda_k$ defined in \eqref{eq:lambda_k} satisfies
		\begin{equation}
			\lambda_{k}(r) 
			= \frac{r^{d-1}\psi(r)^2}{\abs{\phi^{\prime}(r)}} \mu_{k}[F_\w( r^2 (1 - \variabledot) ) ] , 
		\end{equation}
		in other words, it satisfies 
		\begin{equation}
			\frac{r^{d-1}\psi(r)^2}{\abs{\phi^{\prime}(r)}} \int_{\theta \in \S^{d-1}} F_\w( r^2(1 - \theta \cdot \omega) ) P(\theta) \, d\sigma(\theta) 
			= \lambda_k(r) P(\omega) 
		\end{equation}
		for each $k \in \Z_{\geq 0}$, $P \in \HP_k(\R^d)$ and $\omega \in \S^{d-1}$.
		
		Now a little concern is, the factor $\measure{ \S^{d-2} }$ in \eqref{eq:mu in Funk-Hecke}, the surface area of the $(d-2)$-dimensional unit sphere, does not make sense in the case $d = 1$.
		Nevertheless, it is quite easy to see that the following one-dimensional variant of Theorem \ref{thm:Funk-Hecke} holds:
		\begin{theorem} \label{thm:Funk-Hecke 1D}
			Let $k \in \N$, $F \colon \S^{0} \to \C$ and write
			\begin{equation} \label{eq:mu in Func-Hecke 1D}
				\mu_{k}[F] \coloneqq 
				\begin{cases}
					F(1) + F(-1) , & k = 0 , \\
					F(1) - F(-1) , & k = 1 , \\
					0 , & k \geq 2 .
				\end{cases}
			\end{equation}
			Then, for any $P \in \HP_k(\R)$ and $\omega \in \S^{0}$, we have
			\begin{equation} \tag{\ref{eq:Funk-Hecke}} \noeqref{eq:Funk-Hecke}
				\int_{\theta \in \S^{0}} F(\theta \omega) P(\theta) \, d\sigma(\theta) 
				= \mu_{k}[F] P(\omega) . 
			\end{equation}
		\end{theorem}
		Notice that $\lambda_k$ defined in \eqref{eq:lambda_k 1D} satisfies
		\begin{equation}
			\lambda_k(r) 
			=
			\frac{\psi(r)^2}{\abs{\phi^{\prime}(r)}} \mu_k [ F_\w ( r^2 (1 - \variabledot) ) ] .
		\end{equation}
		\subsection{Decomposition lemmas of \texorpdfstring{$\norm{S f}_{L^2(\R^d)}$}{||Sf||} and \texorpdfstring{$\norm{\DiracS f}_{L^2(\R^d, \C^N)}$}{||S~f||}}
		Using results in Section \ref{subsection:Spherical harmonics decomposition} and Section \ref{subsection:Funk-Hecke}, 
		we can obtain following important lemmas.
		\begin{lemma} \label{lem:Sf norm decomposition Schrodinger}
			Suppose that Assumption \ref{assumption} holds.
			Let $f \in L^2(\R^d)$. 
			If 
			\begin{equation}
				f(\xi) = \abs{\xi}^{-(d-1)/2} \sum_{k=0}^{\infty} \sum_{n=1}^{d_k} \Pkn{k}{n}(\xi / \abs{\xi}) \fkn{k}{n}(\abs{\xi}) ,
			\end{equation}
			then we have
			\begin{equation}
				\norm{Sf}_{L^2(\R^{d+1})} = 2 \pi \sum_{k = 0}^{\infty} \int_{0}^{\infty} \lambda_k (r) \mleft( \sum_{n=1}^{d_k} \abs{\fkn{k}{n}(r)}^2 \mright) \, dr .
			\end{equation}
		\end{lemma}
		\begin{lemma} \label{lem:Sf norm decomposition Dirac}
			Suppose that Assumption \ref{assumption} holds.
			Let $f \in L^2(\R^d, \C^N)$ and define $f \pm \in L^2(\R^d, \C^N)$ by
			\begin{equation}
				f_\pm(\xi) \coloneqq \frac{1}{2} \mleft( f(\xi) \pm \frac{1}{\phi(r)} \bigg( m \beta f(\xi) + \sum_{j=1}^{d} \alpha_j \xi_j f(\xi) \bigg) \mright) .
			\end{equation}
			If 
			\begin{equation}
				f_\pm(\xi) = \abs{\xi}^{-(d-1)/2} \sum_{k=0}^{\infty} \sum_{n=1}^{d_k} \Pkn{k}{n}(\xi / \abs{\xi}) \fknpm{k}{n}(\abs{\xi}) ,
			\end{equation}
			then we have
			\begin{equation}
				\norm{\DiracS f}_{L^2(\R^{d+1}, \C^N)} = 2 \pi \sum_{k = 0}^{\infty} \int_{0}^{\infty} \lambda_k (r) \mleft( \sum_{n=1}^{d_k} ( \abs{\fknp{k}{n}(r)}^2 + \abs{\fknm{k}{n}(r)}^2 ) \mright) \, dr .
			\end{equation}
		\end{lemma}	
		We omit the proofs of these lemmas. See \cite[Proof of Theorem 1.1]{BSS} and \cite[Proof of Theorem 2.1]{Ikoma}.
		
		\section{Proof of Main results} \label{section:proof}
		\subsection{Proof of Theorems \ref{thm:1D Schrodinger} and \ref{thm:1D Dirac}: smoothing in the case \texorpdfstring{$d = 1$}{d = 1}} \label{subsection:proof 1D}
		In this section, we prove Theorems \ref{thm:1D Schrodinger} and \ref{thm:1D Dirac}.
		Through out this section, we choose orthonormal bases $\{ \Pk{k} \} \subset \HP_k(\R)$ ($k = 0, 1$) given by
		\begin{equation}
			\Pk{0}(\xi) \coloneqq \frac{1}{\sqrt{2}} , \quad 
			\Pk{1}(\xi) \coloneqq \frac{\xi}{\sqrt{2}} .
		\end{equation}
		
		At first we prove Theorem \ref{thm:1D Schrodinger}. 
		Though the proof is essentially the same as that of \cite[Theorem 1.1]{BSS}, 
		we give details here so that readers can compare it with the proof of Theorem \ref{thm:1D Dirac}. 
		\begin{proof}[Proof of Theorem \ref{thm:1D Schrodinger}]
			Fix $f \in L^2(\R)$ and decompose it as
			\begin{equation}
				f(\xi) = \sum_{k=0, 1} \Pk{k}(\xi / \abs{\xi}) \fk{k}(\abs{\xi}) .
			\end{equation}
			Then Lemma \ref{lem:Sf norm decomposition Schrodinger} implies that
			\begin{align}
				&\norm{Sf}_{L^2(\R^2)}^2 \\
				&= 2 \pi \sum_{k = 0, 1} \int_{0}^{\infty} \lambda_k(r) \abs{ \fk{k}(r) }^2 \, dr
				\\
				&\leq 2 \pi \lambda^* \norm{f}_{L^2(\R)}^2 ,
			\end{align}
			where
			\begin{equation}
				\lambda^* \coloneqq \sup_{r > 0} \max_{k = 0, 1} \lambda_k(r) .
			\end{equation}
			Hence, we have 
			\begin{equation}
				\norm{S}^2 \leq 2 \pi \lambda^* .
			\end{equation}
			
			To see the equality $\norm{S}^2 = 2 \pi \lambda^*$, 
			we will show that for any $\varepsilon > 0$, 
			there exists $f \in L^2(\R) \setminus \{0\}$ 
			such that $\norm{Sf}^{2}_{L^2(\R^2)} \geq 2 \pi ( \lambda^* - \varepsilon ) \norm{f}_{L^2(\R)}^2$.
			Fix $\varepsilon > 0$. 
			Then, by the definition of $\lambda^*$ and the continuity of $\lambda_k$, 
			we can choose $k \in \{0, 1\}$ such that the Lebesgue measure of
			\begin{equation}
				E_k(\varepsilon) \coloneqq \set{r > 0}{ \lambda_k(r) \geq \lambda^* - \varepsilon }	
			\end{equation}
			is nonzero (possibly infinite). 
			Now let $f \in L^2(\R) \setminus \{ 0 \}$ be
			\begin{equation}
				f(\xi) = \sum_{k=0, 1} \Pk{k}( \xi / \abs{\xi} ) \fk{k}(\abs{\xi})
			\end{equation}
			with $\{ \fk{k} \}_{k = 0, 1} \subset L^2(\positiveR)$ satisfying
			\begin{equation}
				\supp \fk{k}(r) \subset E_k(\varepsilon) .
			\end{equation} 
			Then we have
			\begin{align}
				\norm{Sf}^{2}_{L^2(\R^2)}
				&= 2 \pi \int_{0}^{\infty} \lambda_{k}(r) \abs{ \fk{k}(\abs{\xi}) }^2 \, dr \\
				&\geq 2 \pi \int_{0}^{\infty} ( \lambda^* - \varepsilon ) \abs{ \fk{k}(\abs{\xi}) }^2 \, dr \\
				&= 2 \pi ( \lambda^* - \varepsilon ) \norm{f}_{L^2(\R)}^2 ,
			\end{align}
			hence the equality $\norm{S}^2 = 2 \pi \lambda^*$ holds.
			
			By a similar argument, we can show that $\eqref{item:extremiser condition 1D Schrodinger} \implies \eqref{item:extremiser 1D Schrodinger}$.
			Let $f \in L^2(\R)$ be
			\begin{equation}
				f(\xi) = \sum_{k=0, 1} \Pk{k}( \xi / \abs{\xi} ) \fk{k}(\abs{\xi})
			\end{equation}
			and $\{ \fk{k} \}_{k = 0, 1} \subset L^2(\positiveR)$ satisfy
			\begin{equation}
				\supp \fk{k}(r) \subset E_k .
			\end{equation} 
			Then we have
			\begin{align}
				\norm{Sf}_{L^2(\R^2)}^2 
				&= 2 \pi \sum_{k = 0, 1} \int_{0}^{\infty} \lambda_k(r) \abs{ \fk{k}(r) }^2 \, dr
				\\
				&= 2 \pi \sum_{k = 0, 1} \int_{0}^{\infty} \lambda^* \abs{ \fk{k}(r) }^2 \, dr \\
				&= 2 \pi \lambda^* \norm{f}_{L^2(\R)}^2.
			\end{align}
			
			Finally, we show that $\eqref{item:extremiser 1D Schrodinger} \implies \eqref{item:extremiser condition 1D Schrodinger}$.
			Suppose that the equality $\norm{Sf}_{L^2(\R^2)}^2 = 2 \pi \lambda^* \norm{f}_{L^2(\R)}^2$ holds. 
			Then, since
			\begin{equation}
				2\pi \int_{0}^{\infty} \sum_{k=0, 1} (\lambda^*-\lambda_k(r)) \abs{\fk{k}(r)}^2 \, dr = 2 \pi \lambda^* \norm{f}_{L^2(\R)}^2 - \norm{Sf}_{L^2(\R^2)}^2 = 0
			\end{equation}
			and 
			\begin{equation}
				\sum_{k=0, 1} ( \lambda^* - \lambda_{k}(r) ) \abs{\fk{k}(r)}^2 \geq 0 , 
			\end{equation}
			we obtain
			\begin{equation}
				\sum_{k=0, 1} ( \lambda^* - \lambda_{k}(r) ) \abs{\fk{k}(r)}^2 = 0 
			\end{equation}
			for almost every $r > 0$.
			On the other hand, for each $k = 0, 1$, the definition of $E_k$ implies that
			\begin{equation}
				\lambda^* - \lambda_k(r) > 0 
			\end{equation}
			for any $r \in \positiveR \setminus E_k$.
			Therefore, we conclude that
			\begin{equation}
				\abs{\fk{k}(r)}^2 = 0
			\end{equation}
			holds for almost every $r \in \positiveR \setminus E_k$. 
			\end{proof}
		Now we are going to prove Theorem \ref{thm:1D Dirac}, the optimal smoothing estimate for the one dimensional Dirac equation. 
		You may find that the proof is much more complicated than the previous one for Schr\"odinger equations, even in the case $d = 1$. 
		\begin{proof}[Proof of Theorem \ref{thm:1D Dirac}]
			Fix $f \in L^2(\R, \C^2)$ and decompose it as
			\begin{equation}
				f(\xi) = \sum_{k=0, 1} \Pk{k}(\xi / \abs{\xi}) \fk{k}(\abs{\xi}) .
			\end{equation}
			Then, since we have
			\begin{align}
				\fpm(\xi) 
				&= \frac{1}{2} \sum_{k = 0, 1} 
				\mleft( 
				\Pk{k}(\xi / \abs{\xi}) \fk{k}(\abs{\xi}) 
				\pm \frac{1}{\phi(\abs{\xi})} ( m \beta \Pk{k}(\xi / \abs{\xi}) \fk{k}(\abs{\xi}) + \alpha \xi \Pk{k}(\xi / \abs{\xi}) \fk{k}(\abs{\xi}) )
				\mright) \\
				&= \frac{1}{2} \sum_{k = 0, 1} \Pk{k}(\xi / \abs{\xi})
				\mleft( 
				\fk{k}(\abs{\xi}) 
				\pm \frac{1}{\phi(\abs{\xi})} ( m \beta \fk{k}(\abs{\xi}) + \alpha \abs{\xi} \fk{1-k}(\abs{\xi}) )
				\mright) ,
			\end{align}
			Lemma \ref{lem:Sf norm decomposition Dirac} implies that
			\begin{equation} \label{eq:Sf norm by fkp fkm}
				\norm{\DiracS f}_{L^2(\R^2, \C^2)}^2 
				= 2 \pi \sum_{k=0, 1} \int_{0}^{\infty} \frac{1}{2} \lambda_k(r) \mleft( \abs{ \fk{k}(r) }^2 + \frac{1}{\phi(r)^2} \abs{ m \beta \fk{k}(r) + \alpha r \fk{1-k}(r) }^2 \mright) \, dr .
			\end{equation}
			Moreover, using $\alpha^2 = \beta^2 = I_2$, $\adjoint{\alpha} = \alpha$, $\adjoint{\beta} = \beta$ and $\alpha \beta = - \beta \alpha$, 
			the integrand is written as follows:
			\begingroup
			\allowdisplaybreaks
			\begin{align}
				&\quad \frac{1}{2} \sum_{k=0, 1} \lambda_k(r) \mleft( \abs{ \fk{k}(r) }^2 + \frac{1}{\phi(r)^2} \abs{ m \beta \fk{k}(r) + \alpha r \fk{1-k}(r) }^2 \mright)  \\
				&= \frac{1}{2} \sum_{k=0, 1} \lambda_k(r) 
				\mleft( \mleft( 1 + \frac{m^2}{\phi(r)^2} \mright) \abs{ \fk{k}(r) }^2 + \frac{2 m r}{\phi(r)^2} \Re{ \innerproduct{ \beta \fk{k}(r) }{ \alpha \fk{1-k}(r) }_{\C^2} + \frac{r^2}{\phi(r)^2} \abs{ \fk{1-k}(r) }^2 } \mright) \\
				&= \frac{1}{2} \mleft( \mleft( 1 + \frac{m^2}{\phi(r)^2} \mright) \lambda_0(r) + \frac{r^2}{\phi(r)^2} \lambda_{1}(r) \mright) \abs{ \beta \fk{0}(r) }^2 \\
				&+ \frac{mr}{\phi(r)^2} ( \lambda_0(r) - \lambda_1(r) ) ( \innerproduct{ \beta \fk{0}(r) }{ \alpha \fk{1}(r) }_{\C^2} + \innerproduct{ \alpha \fk{1}(r) }{ \beta \fk{0}(r) }_{\C^2} ) \\
				&+ \frac{1}{2} \mleft( \mleft( 1 + \frac{m^2}{\phi(r)^2} \mright) \lambda_1(r) + \frac{r^2}{\phi(r)^2} \lambda_{0}(r) \mright) \abs{ \alpha \fk{1}(r) }^2 \\
				&= \Innerproduct{ 
					\Quadmatrix(r) 
					\begin{pmatrix}
						\beta \fk{0}(r) \\
						\alpha \fk{1}(r)
					\end{pmatrix}
				}
				{
					\begin{pmatrix}
						\beta \fk{0}(r) \\
						\alpha \fk{1}(r)
					\end{pmatrix}
				}_{\C^4} ,
			\end{align}%
			\endgroup
			where		
			\begin{gather}
				\Quadmatrix(r) \coloneqq
				\begin{pmatrix}
					a(r) I_2 & b(r) I_2 / 2 \\
					b(r) I_2 / 2 & c(r) I_2
				\end{pmatrix} , 
				\\
				a(r) \coloneqq \frac{1}{2} \mleft( \mleft( 1 + \frac{m^2}{\phi(r)^2} \mright) \lambda_0(r) + \frac{r^2}{\phi(r)^2} \lambda_{1}(r) \mright) , 
				\\
				b(r) \coloneqq \frac{mr}{\phi(r)^2} ( \lambda_0(r) - \lambda_1(r) ) ,
				\\
				c(r) \coloneqq \frac{1}{2} \mleft( \mleft( 1 + \frac{m^2}{\phi(r)^2} \mright) \lambda_1(r) + \frac{r^2}{\phi(r)^2} \lambda_{0}(r) \mright) .
			\end{gather}
			Now we need to determine the maximal eigenvalue of $\Quadmatrix(r)$ and its associated eigenspace. 
			Note that $\phi(r) = \sqrt{r^2 + m^2}$ implies that
			\begin{align}
				\Quadmatrix(r) 
				&= \frac{1}{2} ( \lambda_0(r) + \lambda_1(r) ) I_4 + 
				\frac{m}{ 2 \phi(r)^2 } ( \lambda_0(r) - \lambda_1(r) )
				\begin{pmatrix}
					m I_2 & r I_2 \\
					r I_2 & - m I_2
				\end{pmatrix} \\
				&= \frac{\psi(r)^2}{\abs{\phi'(r)}} 
				\mleft(  F_\w(0) I_4 + 
				\frac{m}{ \phi(r)^2 } F_\w(2r^2)
				\begin{pmatrix}
					m I_2 & r I_2 \\
					r I_2 & - m I_2
				\end{pmatrix} 
				\mright) .
			\end{align}
			Then, since eigenvalues and associated eigenspaces of the matrix
			\begin{equation}
				\begin{pmatrix}
					m I_2 & r I_2 \\
					r I_2 & - m I_2
				\end{pmatrix}
			\end{equation}
			are $\pm \phi(r)$ and 
			\begin{equation}
				\ker{
					\begin{pmatrix}
						(m \mp \phi(r) ) I_2 & r I_2 \\
						r I_2 & - ( m \mp \phi(r) ) I_2
					\end{pmatrix}
				}
				= 
				\Set{ 
					s
					\begin{pmatrix}
						m \pm \phi(r) \\
						0 \\
						r \\
						0
					\end{pmatrix}
					+ 
					t 
					\begin{pmatrix}
						0 \\
						m \pm \phi(r) \\
						0 \\
						r 
					\end{pmatrix}
				}
				{ s, t \in \C } ,
			\end{equation}
			we conclude that the maximal eigenvalue of $\Quadmatrix(r)$ and its associated eigenspace are
			\begin{equation} \tag*{\eqref{eq:lambda 1D Dirac}}
				\Diraclambda(r) 
				= \frac{\psi(r)^2}{ \abs{\phi'(r)} } \mleft( F_\w(0) + \frac{m}{\phi(r)} \abs{ F_\w(2r^2) } \mright) 
			\end{equation}
			and
			\begin{equation} \tag*{\eqref{eq:eigenspace 1D}}
				\eigenspace(r) = 
				\begin{dcases}
					\C^4, & m F_\w(2r^2) = 0 , \\
					\Set{ 
						s
						\begin{pmatrix}
							m + \phi(r) \\
							0 \\
							r \\
							0
						\end{pmatrix}
						+ 
						t 
						\begin{pmatrix}
							0 \\
							m + \phi(r) \\
							0 \\
							r 
						\end{pmatrix}
					}
					{ s, t \in \C } , 
					& m F_\w(2r^2) > 0 , \\
					\Set{ 
						s
						\begin{pmatrix}
							m - \phi(r) \\
							0 \\
							r \\
							0
						\end{pmatrix}
						+ 
						t 
						\begin{pmatrix}
							0 \\
							m - \phi(r) \\
							0 \\
							r 
						\end{pmatrix}
					}
					{ s, t \in \C } , 
					& m F_\w(2r^2) < 0 ,
				\end{dcases}
			\end{equation}
			respectively.
			Therefore, we have
			\begin{align}
				\norm{\DiracS f}_{L^2(\R^2, \C^2)}^2 
				&= 2 \pi \int_{0}^{\infty} 
				\Innerproduct{ 
					\Quadmatrix(r) 
					\begin{pmatrix}
						\beta \fk{0}(r) \\
						\alpha \fk{1}(r)
					\end{pmatrix}
				}
				{
					\begin{pmatrix}
						\beta \fk{0}(r) \\
						\alpha \fk{1}(r)
					\end{pmatrix}
				}_{\C^4} \, dr \\
				&\leq 2 \pi \int_{0}^{\infty} \Diraclambda(r) ( \abs{ \fk{0}(r) }^2 + \abs{ \fk{1}(r) }^2 ) \, dr \\
				&\leq 2 \pi \Diraclambda^* \norm{f}_{L^2(\R, \C^2)}^2 
			\end{align}
			and hence
			\begin{equation}
				\norm{\DiracS}^2 \leq 2 \pi \Diraclambda^* .
			\end{equation}
			
			To see the equality $\norm{\DiracS}^2 = 2 \pi \Diraclambda^*$, 
			we will show that for any $\varepsilon > 0$, 
			there exists $f \in L^2(\R, \C^2) \setminus \{0\}$ 
			such that $\norm{Sf}^{2}_{L^2(\R^2, \C^2)} \geq 2 \pi ( \Diraclambda^* - \varepsilon ) \norm{f}_{L^2(\R, \C^2)}^2$.
			Fix $\varepsilon > 0$. 
			Then, by the definition of $\Diraclambda^*$ and the continuity of $\Diraclambda$, 
			the Lebesgue measure of
			\begin{equation}
				\DiracE(\varepsilon) \coloneqq \set{r > 0}{ \Diraclambda(r) \geq \Diraclambda^* - \varepsilon }	
			\end{equation}
			is nonzero (possibly infinite). 
			Now let $f \in L^2(\R, \C^2) \setminus \{ 0 \}$ be
			\begin{equation}
				f(\xi) = \sum_{k=0, 1} \Pk{k}( \xi / \abs{\xi} ) \fk{k}(\abs{\xi}) 
			\end{equation}
			with $\{ \fk{k} \}_{k = 0, 1} \subset L^2(\positiveR, \C^2)$ satisfying
			\begin{gather}
				\supp{( \abs{ \fk{0}(r) }^2 + \abs{ \fk{1}(r) }^2 )} \subset \DiracE(\varepsilon), \\
				\begin{pmatrix}
					\beta \fk{0}(r) \\
					\alpha \fk{1}(r)
				\end{pmatrix} 
				\in W(r) , \quad \text{ a.e. } r > 0 .
			\end{gather}
			Then we have
			\begin{align}
				\norm{\DiracS f}_{L^2(\R^2, \C^2)}^2 
				&= 2 \pi \int_{0}^{\infty} 
				\Innerproduct{ 
					\Quadmatrix(r) 
					\begin{pmatrix}
						\beta \fk{0}(r) \\
						\alpha \fk{1}(r)
					\end{pmatrix}
				}
				{
					\begin{pmatrix}
						\beta \fk{0}(r) \\
						\alpha \fk{1}(r)
					\end{pmatrix}
				}_{\C^4} \, dr \\
				&= 2 \pi \int_{0}^{\infty} \Diraclambda(r) ( \abs{ \fk{0}(r) }^2 + \abs{ \fk{1}(r) }^2 ) \, dr \\
				&\geq 2 \pi ( \Diraclambda^* - \varepsilon ) \norm{f}_{L^2(\R, \C^2)}^2 ,
			\end{align}
			hence the equality $\norm{\DiracS}^2 = 2 \pi \lambda^*$ holds.
			
			By a similar argument, we can show that $\eqref{item:extremiser condition 1D Dirac} \implies \eqref{item:extremiser 1D Dirac}$.
			Let $f \in L^2(\R, \C^2)$ be
			\begin{equation}
				f(\xi) = \sum_{k=0, 1} \Pk{k}( \xi / \abs{\xi} ) \fk{k}(\abs{\xi}) 
			\end{equation}
			and $\{ \fk{k} \}_{k = 0, 1} \subset L^2(\positiveR, \C^2)$ satisfy
			\begin{gather}
				\supp{( \abs{ \fk{0}(r) }^2 + \abs{ \fk{1}(r) }^2 )} \subset \DiracE , \\
				\begin{pmatrix}
					\beta \fk{0}(r) \\
					\alpha \fk{1}(r)
				\end{pmatrix} 
				\in \eigenspace(r) , \quad \text{ a.e. } r > 0 .
			\end{gather}
			Then we have
			\begin{align}
				\norm{Sf}_{L^2(\R^2, \C^2)}^2 
				&= 2 \pi \int_{0}^{\infty} 
				\Innerproduct{ 
					\Quadmatrix(r) 
					\begin{pmatrix}
						\beta \fk{0}(r) \\
						\alpha \fk{1}(r)
					\end{pmatrix}
				}
				{
					\begin{pmatrix}
						\beta \fk{0}(r) \\
						\alpha \fk{1}(r)
					\end{pmatrix}
				}_{\C^4} \, dr
				\\
				&= 2 \pi \int_{0}^{\infty} \Diraclambda(r) ( \abs{ \fk{0}(r) }^2 + \abs{ \fk{1}(r) }^2 ) \, dr \\
				&= 2 \pi \Diraclambda^* \norm{f}_{L^2(\R, \C^2)}^2.
			\end{align}
			
			Finally, we show that $\eqref{item:extremiser 1D Dirac} \implies \eqref{item:extremiser condition 1D Dirac}$.
			Suppose that the equality $\norm{\DiracS f}_{L^2(\R^2, \C^2)}^2 = 2 \pi \Diraclambda^* \norm{f}_{L^2(\R, \C^2)}^2$ holds. 
			Then, since
			\begin{align}
				&\quad 2\pi \int_{0}^{\infty} \mleft( \Diraclambda^* ( \abs{\fk{0}(r)}^2 + \abs{\fk{1}(r)}^2 ) - \Innerproduct{ 
					\Quadmatrix(r) 
					\begin{pmatrix}
						\beta \fk{0}(r) \\
						\alpha \fk{1}(r)
					\end{pmatrix}
				}
				{
					\begin{pmatrix}
						\beta \fk{0}(r) \\
						\alpha \fk{1}(r)
					\end{pmatrix}
				}_{\C^4} \mright)  \, dr \\
				&= 2 \pi \Diraclambda^* \norm{f}_{L^2(\R, \C^2)}^2 - \norm{\DiracS f}_{L^2(\R^2, \C^2)}^2 \\
				&= 0
			\end{align}
			and 
			\begin{equation}
				\Diraclambda^* ( \abs{\fk{0}(r)}^2 + \abs{\fk{1}(r)}^2 ) - \Innerproduct{ 
					\Quadmatrix(r) 
					\begin{pmatrix}
						\beta \fk{0}(r) \\
						\alpha \fk{1}(r)
					\end{pmatrix}
				}
				{
					\begin{pmatrix}
						\beta \fk{0}(r) \\
						\alpha \fk{1}(r)
					\end{pmatrix}
				}_{\C^4} \geq 0 ,
			\end{equation}
			we obtain 
			\begin{gather}
				\begin{pmatrix}
					\beta \fk{0}(r) \\
					\alpha \fk{1}(r)
				\end{pmatrix} \in \eigenspace(r)  , \\
				\begin{aligned}
					& \Diraclambda^* ( \abs{\fk{0}(r)}^2 + \abs{\fk{1}(r)}^2 ) - \Innerproduct{ 
						\Quadmatrix(r) 
						\begin{pmatrix}
							\beta \fk{0}(r) \\
							\alpha \fk{1}(r)
						\end{pmatrix}
					}
					{
						\begin{pmatrix}
							\beta \fk{0}(r) \\
							\alpha \fk{1}(r)
						\end{pmatrix}
					}_{\C^4} \\
					&= ( \Diraclambda^* - \Diraclambda(r) ) ( \abs{\fk{0}(r)}^2 + \abs{\fk{1}(r)}^2 ) = 0 
				\end{aligned}
			\end{gather}
			for almost every $r > 0$.
			On the other hand, the definition of $\DiracE$ implies that 
			\begin{equation}
				\Diraclambda^* - \Diraclambda(r) > 0 
			\end{equation}
			for any $r \in \positiveR \setminus \DiracE$.
			Therefore, we conclude that 
			\begin{equation}
				\abs{\fk{0}(r)}^2 + \abs{\fk{1}(r)}^2 = 0 
			\end{equation}
			holds for almost every $r \in \positiveR \setminus \DiracE$.
			\end{proof}
		
		\subsection{Proof of Theorems \ref{thm:radial Schrodinger} and \ref{thm:radial Dirac}: smoothing with radial data} \label{subsection:proof radial}
		In this section, we prove Theorems \ref{thm:radial Schrodinger} and \ref{thm:radial Dirac}.
		Through out this section, we choose orthonormal bases $\{ \Pk{0} \} \subset \HP_0(\R^d)$ and $\{ \Pkn{1}{n} \}_{1 \leq n \leq d} \subset \HP_1(\R^d)$ given by
		\begin{equation}
			\Pk{0}(\xi) \coloneqq (d V_d)^{-1/2} , \quad 
			\Pkn{1}{n}(\xi) \coloneqq (V_d)^{-1/2} \xi_n ,
		\end{equation}
				where $V_d$ denotes the volume of the $d$-dimensional unit ball.
		At first we give the proof of Theorem \ref{thm:radial Schrodinger} as a comparison of that of Theorem \ref{thm:radial Dirac}.
		\begin{proof}[Proof of Theorem \ref{thm:radial Schrodinger}]
			Fix $f \in L^2_\rad(\R^d)$ and write
			\begin{equation}
				f(\xi) = \abs{\xi}^{-(d-1)/2} \Pk{0}(\xi / \abs{\xi}) \fk{0} (\abs{\xi}) .
			\end{equation}
			Then Lemma \ref{lem:Sf norm decomposition Schrodinger} implies that 
			\begin{align}
				&\norm{Sf}_{L^2(\R^{d+1})}^2 \\
				&= 2 \pi \int_{0}^{\infty} \lambda_\rad(r) \abs{ \fk{0}(r) }^2 \, dr
				\\
				&\leq 2 \pi \lambda_\rad^* \norm{f}_{L^2(\R^d)}^2 ,
			\end{align}
			where
			\begin{equation}
				\lambda_\rad^* \coloneqq \sup_{r > 0} \lambda_\rad(r) .
			\end{equation}
			Hence, we have 
			\begin{equation}
				\norm{S}_\rad^2 \leq 2 \pi \lambda_\rad^* .
			\end{equation}
			
			To see the equality $\norm{S}_\rad^2 = 2 \pi \lambda^*$, 
			we will show that for any $\varepsilon > 0$, 
			there exists $f \in L_\rad^2(\R^d) \setminus \{0\}$ 
			such that $\norm{Sf}^{2}_{L^2(\R^{d+1})} \geq 2 \pi ( \lambda_\rad^* - \varepsilon ) \norm{f}_{L^2(\R^d)}^2$.
			Fix $\varepsilon > 0$. 
			Then, by the definition of $\lambda_\rad^*$ and the continuity of $\lambda_0$, 
			the Lebesgue measure of
			\begin{equation}
				E_\rad(\varepsilon) \coloneqq \set{r > 0}{ \lambda_\rad(r) \geq \lambda_\rad^* - \varepsilon }	
			\end{equation}
			is nonzero (possibly infinite). 
			Now let $f \in L_\rad^2(\R^d)$ be
			\begin{equation}
				f(\xi) = \Pk{0}( \xi / \abs{\xi} ) \fk{0}(\abs{\xi}) 
			\end{equation}
			with $\fk{0} \in L^2(\positiveR)$ satisfying
			\begin{equation}
				\supp \fk{0} \subset  E_\rad(\varepsilon) .
			\end{equation} 
			Then we have
			\begin{align}
				\norm{Sf}^{2}_{L^2(\R^{d+1})}
				&= 2 \pi \int_{0}^{\infty} \lambda_\rad(r) \abs{ \fk{0}(\abs{\xi}) }^2 \, dr \\
				&\geq 2 \pi \int_{0}^{\infty} ( \lambda_\rad^* - \varepsilon ) \abs{ \fk{0}(\abs{\xi}) }^2 \, dr \\
				&= 2 \pi ( \lambda_\rad^* - \varepsilon ) \norm{f}_{L^2(\R^d)}^2 ,
			\end{align}
			hence the equality $\norm{S}_\rad^2 = 2 \pi \lambda_\rad^*$ holds.
			
			By a similar argument, we can show that $\eqref{item:extremiser condition radial Schrodinger} \implies \eqref{item:extremiser radial Schrodinger}$.
			Let $f \in L_\rad^2(\R^d)$ be
			\begin{equation}
				f(\xi) = \Pk{0}(\xi/ \abs{\xi}) \fk{0}(\abs{\xi})
			\end{equation}
			and $\fk{0} \in L^2(\positiveR)$ satisfy
			\begin{equation}
				\supp{\fk{0}} \subset E_\rad .
			\end{equation}
			Then we have
			\begin{align}
				\norm{Sf}_{L^2(\R^{d+1})}^2 
				&= 2 \pi \int_{0}^{\infty} \lambda_\rad(r) \abs{ \fk{0}(r) }^2 \, dr
				\\
				&= 2 \pi \int_{0}^{\infty} \lambda_\rad^* \abs{ \fk{0}(r) }^2 \, dr \\
				&= 2 \pi \lambda_\rad^* \norm{f}_{L^2(\R^d)}^2.
			\end{align}
			
			Finally, we show that $\eqref{item:extremiser radial Schrodinger} \implies \eqref{item:extremiser condition radial Schrodinger}$.
			Suppose that the equality $\norm{Sf}_{L^2(\R^{d+1})}^2 = 2 \pi \lambda^* \norm{f}_{L^2(\R^d)}^2$ holds. 
			Then, since
			\begin{equation}
			2\pi \int_{0}^{\infty} (\lambda_\rad^*-\lambda_\rad(r)) \abs{\fk{0}(r)}^2 \, dr = 2 \pi \lambda_\rad^* \norm{f}_{L^2(\R^d)}^2 - \norm{Sf}_{L^2(\R^{d+1})}^2 = 0
			\end{equation}
			and 
			\begin{equation}
				( \lambda_\rad^* - \lambda_\rad(r) ) \abs{\fk{0}(r)}^2 \geq 0 , 
			\end{equation}
			we obtain
			\begin{equation}
				( \lambda_\rad^* - \lambda_\rad(r) ) \abs{\fk{0}(r)}^2 = 0 
			\end{equation}
			for almost every $r > 0$.
			On the other hand, the definition of $E_0$ implies that
			\begin{equation}
				\lambda_\rad^* - \lambda_\rad(r) > 0 
			\end{equation}
			for any $r \in \positiveR \setminus E_\rad$.
			Therefore, we conclude that
			\begin{equation}
				\abs{\fk{0}(r)}^2 = 0
			\end{equation}
			holds for almost every $r \in \positiveR \setminus E_\rad$. 
			\end{proof}
		Next we prove Theorem \ref{thm:radial Dirac}. 
		\begin{proof}[Proof of Theorem \ref{thm:radial Dirac}]
			Fix $f \in L^2_\rad(\R^d, \C^N)$ and write
			\begin{equation}
				f(\xi) = \abs{\xi}^{-(d-1)/2} \Pk{0}(\xi / \abs{\xi}) \fk{0} (\abs{\xi}) .
			\end{equation}
			Then, since we have
			\begin{align}
				f_\pm f(\xi) 
				=  \frac{\abs{\xi}^{-(d-1)/2}}{2}
				\mleft( 
				\mleft( I_N \pm \frac{m}{\phi(\abs{\xi})} \beta \mright) \Pk{0}(\xi / \abs{\xi}) \fk{0}(\abs{\xi}) \pm \frac{\abs{\xi}}{\sqrt{d} \phi(\abs{\xi})} \sum_{j=1}^{d} \Pkn{1}{j}(\xi / \abs{\xi}) \alpha_j \fk{0}(\abs{\xi}) \mright) ,
			\end{align}
			Lemma \ref{lem:Sf norm decomposition Dirac} implies that
			\begin{align}
				&\quad \norm{\DiracS f}^{2}_{L^2(\R^{d+1},\C^N)} \\
				&=
				2\pi \int_{0}^{\infty} \lambda_0(r) \mleft( \Abs{ \frac{1}{2} \mleft( I_N + \frac{m}{\phi(r)} \beta \mright) \fk{0}(r) }^2 + \Abs{ \frac{1}{2} \mleft( I_N - \frac{m}{\phi(r)} \beta \mright) \fk{0}(r) }^2 \mright) \, dr \\
				&\quad + 2\pi \sum_{j=1}^{d} \int_{0}^{\infty} 2 \lambda_1(r) \Abs{ \frac{r}{2 \sqrt{d} \phi(r)} \alpha_j \fk{0}(r) }^2 \, dr \\
				&= 
				2\pi \int_{0}^{\infty} \mleft( \frac{1}{2} \mleft( 1 + \frac{m^2}{\phi(r)^2} \mright) \lambda_0(r)  +  \frac{r^2}{2 \phi(r)^2} \lambda_1(r) \mright) \abs{ \fk{0}(r) }^2 \, dr \\
				&= 2 \pi \int_{0}^{\infty} \Diraclambda_\rad(r) \abs{ \fk{0}(r) }^2 \, dr \\
				&\leq 2 \pi \Diraclambda_\rad^* \norm{f}_{L^2(\R^d, \C^N)}, 
			\end{align}
			and thus $\norm{\DiracS}_\rad \leq 2 \pi \Diraclambda_\rad^*$.
			
			To see the equality $\norm{\DiracS}_\rad^2 = 2 \pi \Diraclambda_\rad^*$, 
			we will show that for any $\varepsilon > 0$, 
			there exists $f \in L^2_\rad(\R^d, \C^N) \setminus \{0\}$ 
			such that $\norm{\DiracS f}^{2}_{L^2(\R^{d+1},\C^N)} \geq 2 \pi ( \Diraclambda_\rad^* - \varepsilon ) \norm{f}_{L^2(\R^d, \C^N)}^2$.
			Fix $\varepsilon > 0$. 
			Then, by the definition of $\Diraclambda_\rad^*$, 
			the Lebesgue measure of
			\begin{equation}
				\DiracE_{\rad}(\varepsilon) \coloneqq \set{r > 0}{ \lambda_\rad(r) \geq \Diraclambda_\rad^* - \varepsilon }	
			\end{equation}
			is nonzero (possibly infinite). 
			Now let $f \in L^2(\positiveR, \C^N) \setminus \{0\}$ be
			\begin{equation}
				f(\xi) = \abs{\xi}^{-(d-1)/2} \Pk{0}(\xi / \abs{\xi}) \fk{0}(\abs{\xi}) 
			\end{equation}
			with $\fk{0} \in L^2(\positiveR, \C^N)$ satisfying
			\begin{equation}
				\supp{ \fk{0} } \subset \DiracE_{\rad} .
			\end{equation}
			Then we have
			\begin{align}
				\norm{\DiracS f}^{2}_{L^2(\R^{d+1},\C^N)}
				&= 2 \pi \int_{0}^{\infty} \Diraclambda_\rad(r) \abs{\fk{0}(r)}^2 \, dr \\
				&\geq 2 \pi \int_{0}^{\infty} ( \Diraclambda_\rad^* - \varepsilon ) \abs{\fk{0}(r)}^2 \, dr \\
				&= 2 \pi ( \Diraclambda_\rad^* - \varepsilon ) \norm{f}_{L^2(\R^d, \C^N)}^2 ,
			\end{align}
			hence the equality $\norm{\DiracS}_\rad^2 = 2 \pi \Diraclambda_\rad^*$ holds.
			
			By a similar argument, we can show that $\eqref{item:extremiser condition radial Dirac} \implies \eqref{item:extremiser radial Dirac}$. 
			Let $f \in L^2_\rad(\R^d, \C^N)$ be
			\begin{equation}
				f(\xi) = \abs{\xi}^{-(d-1)/2} \Pk{0}(\xi / \abs{\xi}) \fk{0}(\abs{\xi}) 
			\end{equation}
			and $\fk{0} \in L^2(\positiveR, \C^N)$ satisfy
			\begin{equation}
				\supp{ \fk{0} } \subset \DiracE_{\rad} .
			\end{equation}
			Then we have
			\begin{align}
				\norm{\DiracS f}^{2}_{L^2(\R^{d+1},\C^N)}
				&= 2 \pi \int_{0}^{\infty} \Diraclambda_\rad(r) \abs{\fk{0}(r)}^2 \indicator{\DiracE_{\rad}}(r) \, dr \\
				&\geq 2 \pi \int_{0}^{\infty} \Diraclambda_\rad^* \abs{\fk{0}(r)}^2 \indicator{\DiracE_{\rad}}(r) \, dr \\
				&= 2 \pi \Diraclambda_\rad^* \norm{f}_{L^2(\R^d, \C^N)}^2 .
			\end{align}
			Finally, we show that $\eqref{item:extremiser radial Dirac} \implies \eqref{item:extremiser condition radial Dirac}$.
			Suppose that the equality $\norm{\DiracS f}_{L^2(\R^{d+1}, \C^N)}^2 = 2 \pi \Diraclambda_\rad^* \norm{ f }_{L^2(\R^d, \C^N)}^2$ holds.
			Then, since
			\begin{equation}
			2 \pi \int_{0}^{\infty} ( \Diraclambda_\rad^* - \Diraclambda_\rad(r) ) \abs{\fk{0}(r)}^2  \, dr = 2 \pi \Diraclambda_\rad^* \norm{ f }_{L^2(\R^d, \C^N)}^2 - \norm{\DiracS f}_{L^2(\R^{d+1}, \C^N)}^2  = 0
			\end{equation}
			and 
			\begin{equation}
				( \Diraclambda_\rad^* - \Diraclambda_\rad(r) ) \abs{\fk{0}(r)}^2 \geq 0 , 
			\end{equation}
			we obtain
			\begin{equation}
				( \Diraclambda_\rad^* - \Diraclambda_\rad(r) ) \abs{\fk{0}(r)}^2 = 0 
			\end{equation}
			for almost every $r > 0$.
			On the other hand, by the definition of $\DiracE_{\rad}$, we have
			\begin{equation}
				\Diraclambda_\rad^* - \Diraclambda_\rad(r) > 0
			\end{equation}
			for any $r \in \positiveR \setminus \DiracE_{\rad}$. Therefore, we conclude that
			\begin{equation}
				\abs{\fk{0}(r)}^2 = 0 
			\end{equation}
			holds for almost every $r \in \positiveR \setminus \DiracE_{\rad}$.
			\end{proof}
\subsection{Proof of Theorem \ref{thm:explicit value}: explicit values}
			Theorem \ref{thm:explicit value} easily follows from the following lemma:
		\begin{lemma}[\cite{BS}] \label{lem:BS 1.6}
			Let $d \geq 2$, $1 < s < d$, and suppose that $( w, \psi, \phi )$ satisfy
			\begin{equation}
				w(x) = r^{-s}, \quad \psi(r)^2 = r^{1-s} \abs{ \phi'(r) } .
			\end{equation}
			Then we have
			\begin{equation}
				\lambda_k(r) = c_k = 2^{1-s} (2\pi)^{d} \frac{ \Gamma(s-1) \Gamma((d-s)/2 + k) }{ ( \Gamma(s/2) )^2 \Gamma( (d + s)/2 + k - 1 ) } .
			\end{equation}
			Furthermore, $\{ c_k \}_{k \in \N}$ is strictly decreasing.
			For example, if $d \geq 3$ and $s = 2$, then
			\begin{equation}
				\lambda_k(r) = c_k = \frac{(2\pi)^{d}}{d + 2k - 2} .
			\end{equation}
		\end{lemma}
		The proof of Lemma \ref{lem:BS 1.6} can be found in \cite[Proof of Theorem 1.6]{BS}.
		\begin{proof}[Proof of Theorem \ref{thm:explicit value}]
			By Lemma \ref{lem:BS 1.6}, we have
			\begin{gather}
				\lambda_k(r) = c_k , \\
				\Diraclambda_\rad(r) = \frac{1}{2}\mleft( c_0 + c_1 + \frac{m^2}{r^2 + m^2}( c_0 - c_1 ) \mright)
			\end{gather}
			and so that 
			\begin{gather}
				\sup_{k \in \N} \sup_{r > 0} \lambda_{k}(r) = \sup_{r > 0} \lambda_{k}(r) = c_0, \\
				\sup_{r>0} \Diraclambda_\rad(r)
				= \begin{cases}
					(c_0 + c_1)/2 , & m = 0 , \\
					c_0 , & m > 0 .
				\end{cases}
			\end{gather}
			Therefore, by Corollary \ref{cor:lower and upper bounds}, we conclude that 
			\begin{equation}
				\norm{\DiracS}^2 = 2 \pi c_0
			\end{equation}
			holds when $m > 0$. 
			In fact, in this case we have
			\begin{equation}
		 	\norm{S}^2 = \norm{S}_\rad^2 = \norm{\DiracS}^2 = \norm{\DiracS}_\rad^2 = 2 \pi c_0 . \quad \qedhere 
			\end{equation}
		\end{proof}
\section*{Acknowledgments}
			The authors would like to thank Mitsuru Sugimoto (Nagoya University) for valuable discussions.
		
		\bibliographystyle{spmpsci}
		\bibliography{bib_ikoma_suzuki}
	\end{document}